\documentclass[11pt]{amsart}
\newcommand{\K}{\mathcal{K}}
\newcommand{\R}{\mathbb{R}}
\newcommand{\N}{\mathbb{N}}
\newtheorem{theorem}{Theorem}
\newtheorem{lemma}{Lemma}

\usepackage{amssymb}

\renewcommand{\le}{\leqslant}

\renewcommand{\ge}{\geqslant}

\newcommand{\ddd}{ \,\frac{d\mu(\vartheta)\; dr}{|r|^{1+2s}} }
\newcommand{\XXX}{ S^{n-1}\times\R }

\begin{document}

\title[Nonlocal semilinear equations]{A rigidity result\\
for nonlocal semilinear equations}
\thanks{The authors have been supported by the
ERC grant 277749 ``EPSILON Elliptic
Pde's and Symmetry of Interfaces and Layers for Odd Nonlinearities''}

\author{Alberto Farina}

\address{Alberto Farina:
LAMFA -- CNRS UMR 6140 --
Universit\'e de Picardie Jules Verne --
Facult\'e des Sciences --
33, rue Saint-Leu --
80039 Amiens CEDEX 1, France --
Email:
{\tt alberto.farina@u-picardie.fr}  
}

\author{Enrico Valdinoci}

\address{Enrico Valdinoci:
Weierstra{\ss}-Institut f\"ur Angewandte Analysis und Stochastik --
Mohrenstra{\ss}e, 39 --
10117 Berlin, Germany -- and --
Universit\`a degli Studi di Milano -- Dipartimento di Matematica --
Via Cesare Saldini, 50 -- 20133 Milano, Italy --
Email:
{\tt enrico@math.utexas.edu} 
}

\begin{abstract}
We consider a possibly anisotropic integro-differential semilinear equation,
run by a nondecreasing nonlinearity.
We prove that if the solution grows at infinity less than the order
of the operator, then it must be affine (possibly constant).
\end{abstract}


\maketitle

\section{Introduction}

It dates back to Liouville and Cauchy
in~1844 that bounded harmonic functions are constant.
Several generalizations of this result appeared in the literature,
also involving nonlinear equations and more general growth of the solution
at infinity (see~\cite{handbook} for a detailed review of this topic).
\medskip

The purpose of this note is to obtain a rigidity result
for integro-differential semilinear equations
of fractional order~$2s$, with~$s\in(0,1)$.
\medskip

We recall that fractional integro-differential
operators are a classical topic in analysis, whose
study arises in different fields, including harmonic analysis~\cite{St},
partial differential equations~\cite{C} and probability~\cite{Be}. 
The study of these operators has also relevance in concrete
situations,
in view of the related real-world applications,
such as quantum mechanics~\cite{FLl},
water waves~\cite{CSS}, meteorology~\cite{CV},
crystallography~\cite{G},
biology~\cite{AAVV}, finance~\cite{Sc}
and high technology~\cite{ZL},
just to name a few.

The type of integro-differential
operators that we consider here are of the form
\begin{equation}\label{OP1}
{\mathcal{I}}u(x):=
\int_{\XXX}
\big(u(x+\vartheta r)+u(x-\vartheta r)-2u(x)\big)\ddd.\end{equation}
In this notation, $r\in\R$ gets integrated by the usual
Lebesgue measure~$dr$, while~$\vartheta\in S^{n-1}$
is integrated by a measure~$d\mu(\vartheta)$,
which is called ``spectral measure''
in jargon, and that satisfy the following
nondegeneracy assumptions: there exist~$\lambda$, $\Lambda\in(0,+\infty)$ such that
\begin{equation}\label{spectral}
\inf_{\nu\in S^{n-1}} \int_{S^{n-1}}|\nu\cdot\vartheta|^{2s}\,d\mu(\vartheta)
\ge\lambda \;{\mbox{ and }}\;
\mu(S^{n-1})\le \Lambda.
\end{equation}
Particularly famous cases of spectral measures are the ones
induced by singular kernels,
i.e. when~$d\mu(\vartheta)=\K_0(\vartheta)\,d{\mathcal{H}}^{n-1}(\vartheta)$,
with~$0<\inf_{S^{n-1}} \K_0\le\sup_{S^{n-1}} \K_0<+\infty$.
Notice that in this particular case the spectral measure
is absolutely continuous with respect to the standard Hausdorff measure
on~$S^{n-1}$, and
the operator in~\eqref{OP1} comes from the
integration against the homogeneous kernel
\begin{equation}\label{UI9011} 
\K(y):=|y|^{-n-2s} \K_0\left(\frac{y}{|y|}\right),\end{equation}
in the sense that
\begin{equation}\label{UI9012} {\mathcal{I}}u(x)=
\int_{\R^n}
\big(u(x+y)+u(x-y)-2u(x)\big)\,\K(y)\,dy.\end{equation}
Of course, the case~$\K$ equal to constant boils down to the
fractional Laplacian,
i.e. to the case in which, up to a normalization factor,~${\mathcal{I}}=
-(-\Delta)^s$.
\medskip

The literature has recently shown an increasing effort
towards the study of these types of anisotropic
operators, see e.g.~\cite{KRS, KS, ROS, ROV, ROSV}.
It is worth recalling that the picture provided
by the general operator in~\eqref{OP1} is often
quite special when compared with
the isotropic case, and sometimes even surprising:
for instance, a complete regularity
theory in the setting of~\eqref{OP1}--\eqref{spectral}
does not hold, and explicit counterexamples can be
constructed, see~\cite{ROV}.
\medskip

We will consider the equation~${\mathcal{I}}u=f(u)$.
This type of equations is often called ``semilinear''
since the nonlinearity only depends on the values of the solution
itself (for these reasons, solutions of semilinear equations
may satisfy better geometric properties than solutions
of arbitrary equations).

Our main result states that if~$f$ is
nondecreasing, then
solutions of~${\mathcal{I}}u=f(u)$ whose growth
at infinity is bounded by~$|x|^\kappa$, with~$\kappa$
less than the order of operator, must be necessarily affine
(and, in fact, constant when the nonlinearity is nontrivial).
More precisely, we have:

\begin{theorem}\label{Liouville}
Let~$f\in C(\R)$ be nondecreasing.
Let~$u\in C^2(\R^N)$ be a solution of
\begin{equation}\label{U}
{\mathcal{I}}u(x)=f(u(x)) \ {\mbox{ for any }} \ x\in\R^n.\end{equation}
Assume that
\begin{equation}\label{K} |u(x)|\le K\,(1+|x|^\kappa),\end{equation}
for some~$K\ge0$ and~$\kappa\in[0,2s)$. 

Then the following classification holds true:
\begin{itemize}
\item[{(i)}] if~$f$ is not identically zero, then~$u$ is constant, say~$u(x)=c$
for any~$x\in\R^n$, and~$f(c)=0$;
\item[{(ii)}]  if~$f$ is identically zero, then~$u$ is an affine
function, say~$u(x)=\varpi\cdot x+c$, for some~$\varpi\in\R^n$
and~$c\in\R$; in this case, if additionally~$\kappa<1$,
then~$\varpi=0$ and~$u$ is constant.
\end{itemize}
\end{theorem}

We remark that Theorem~\ref{Liouville}
holds true here
for the very general integro-differential
operator in~\eqref{OP1}--\eqref{spectral}. Moreover,
as far as we know, Theorem~\ref{Liouville}
is new even in the case of regular spectral measure
in~\eqref{UI9011}--\eqref{UI9012}, and, perhaps quite surprisingly,
even in the isotropic case of the fractional Laplacian.
\medskip

On the other hand, when~${\mathcal{I}}$ is replaced by the Laplacian
(which is formally the above case with~$s=1$)
Theorem~\ref{Liouville} is a well known
result in the framework of classical Liouville-type theorems:
see for instance~\cite{handbook, serrin}.
As a matter of fact, the counterpart of~\eqref{U}
in the classical case, is the semilinear equation~$\Delta u=f(u)$,
set in the whole of~$\R^n$. This equation has been extensively
studied, also in connection with phase transition models,
see e.g.~\cite{M, Br} and references therein.
Its fractional analogue has also physical relevance,
since it appears, for instance, in the study of phase transitions
arising from long-range interactions and in the dynamics
of atom dislocations in crystal, see e.g.~\cite{CSM, SV, CS, GM, MP, DPV}.
\medskip

The strategy of the proof of
Theorem~\ref{Liouville} is to show that~$f(u(x))$ must be 
identically equal to zero in any case, therefore~$u$ is a solution
of~${\mathcal{I}}u=0$, and this will allow us to
use a Liouville-type theorem in order to obtain the desired classification.
To this goal, one uses the subcritical growth of
the solution~$u$ to compare with suitable barriers.
Of course, the construction of the appropriate barriers
is the main novelty with respect to the classical case,
since the nonlocality of the operator mostly comes into play
in this framework.\medskip

It is also worth pointing out that Theorem~\ref{Liouville}
has a natural, and simple, generalization
that deals with the case in which~\eqref{K} is replaced by a one-side
inequality. In this spirit, we present the following result:

\begin{theorem}\label{Liouville.2}
Let~$f\in C(\R)$ be nondecreasing.
Let~$u\in C^2(\R^N)$ be a solution of
$$ {\mathcal{I}}u(x)=f(u(x)) \ {\mbox{ for any }} \ x\in\R^n.$$
Then, if
$$ u(x)\le K\,(1+|x|^\kappa),$$
for some~$K\ge0$ and~$\kappa\in[0,2s)$, we have that
$$ {\mathcal{I}}u(x)\le0 \ {\mbox{ for any }} \ x\in\R^n.$$ 
Similarly, if
$$ u(x)\ge -K\,(1+|x|^\kappa),$$            
for some~$K\ge0$ and~$\kappa\in[0,2s)$, we have that
$$ {\mathcal{I}}u(x)\ge0 \ {\mbox{ for any }} \ x\in\R^n.$$
\end{theorem}

We also point out that the regularity
assumptions of~$u$
in Theorems~\ref{Liouville} and~\ref{Liouville.2} were taken for
the sake of simplicity and, in concrete cases, one does not
need~$u$ to be smooth to start with
(for instance, in the setting of~\cite{KRS} one can deal
with viscosity solutions and in the setting of~\cite{ROS}
one can deal with weak solutions).\medskip

The rest of the paper is organized as follows. 
First, in Section~\ref{S:T}
we collect some preliminary integral computations that will be
used in Section~\ref{S:B} to construct a useful barrier.
Roughly speaking, this barrier replaces the classical paraboloid
in our nonlocal framework (of course, checking the
properties of the paraboloid in the classical case is much simpler
than constructing barriers in nonlocal cases).

The proofs of Theorems~\ref{Liouville} and~\ref{Liouville.2}
occupy
Section~\ref{P:P}.

\section{Toolbox}\label{S:T}

Below are some preliminary integral computations, needed to
construct a suitable barrier in Section~\ref{S:B}.
For convenience, we use the notation
\begin{eqnarray*}
&&{\mathcal{I}}_1 v(x):=\int_{S^{n-1}\times(-1,1)} 
\big(v(x+\vartheta r)+v(x-\vartheta r)-2v(x)\big)\ddd
\\{\mbox{and }}&&
{\mathcal{I}}_2 v(x):=\int_{S^{n-1}\times(\R\setminus (-1,1))}
\big(v(x+\vartheta r)+v(x-\vartheta r)-2v(x)\big)\ddd,
\end{eqnarray*}
in order to distinguish the integration performed when~$|r|<1$
from the one when~$|r|\ge1$.

\subsection{Estimates near the origin}
Here we estimate~${\mathcal{I}}_1 v$ and~${\mathcal{I}}_2 v$
near the origin
according to the following
Lemmata~\ref{lemma P1} and~\ref{lemma P2}:

\begin{lemma}\label{lemma P1}
Let~$v\in C^2(B_3)$. Then, for any~$x\in B_1$,
$$ {\mathcal{I}}_1 v(x)\le C,$$
for some~$C>0$ possibly depending on~$n$, $s$, $\Lambda$
and~$\|v\|_{C^2(B_2)}$.
\end{lemma}

\begin{proof} If $x$, $y\in B_1$
we obtain from a Taylor expansion that
$$ \big| v(x+y)+v(x-y)-2v(x)\big|\le \|D^2 v\|_{L^\infty(B_2)}\,|y|^2.$$
hence, by integration, we get
\begin{eqnarray*} {\mathcal{I}}_1 v(x)&\le& 
\int_{S^{n-1}\times(-1,1)} \|D^2 v\|_{L^\infty(B_2)}\,|r|^2\ddd\\
&\le &\Lambda \, \|D^2 v\|_{L^\infty(B_2)}\,\int_0^1 |r|^{1-2s}\,dr,\end{eqnarray*}
thanks to~\eqref{spectral}, which gives the desired result.
\end{proof}

\begin{lemma}\label{lemma P2}
Let
\begin{equation}\label{g}
\gamma\in(0, 2s).\end{equation}
Let~$v:\R^n\rightarrow[0,+\infty)$ be a measurable
function such that~$v(x)\le |x|^\gamma$ for any~$x\in\R^n$. 
Then, for any~$x\in B_1$,
$$ {\mathcal{I}}_2 v(x)\le C,$$
for some~$C>0$ possibly depending on~$n$, $s$, $\Lambda$
and~$\gamma$.
\end{lemma}

\begin{proof} Let~$x\in B_1$ and~$y\in\R^n\setminus B_1$.
Then~$|x|\le1\le|y|$ and so
\begin{eqnarray*}
&& |v(x+y)+v(x-y)-v(x)|\\ &&\qquad\le |v(x+y)|+|v(x-y)|+|v(x)|\\ &&\qquad
\le |x+ y|^\gamma+|x-y|^\gamma
+|x|^\gamma
\\ &&\qquad\le 2(|x|+|y|)^\gamma+|x|^\gamma\\ &&\qquad
\le (2^{\gamma+1}+1) \,|y|^\gamma.\end{eqnarray*}
So, we integrate, we recall~\eqref{spectral} and we see that
\begin{eqnarray*} {\mathcal{I}}_2 v(x)&\le&
\int_{S^{n-1}\times(\R^n\setminus (-1,1))}
(2^{\gamma+1}+1) \,|r|^\gamma \ddd\\&\le&
2(2^{\gamma+1}+1)\,\Lambda\,\int_1^{+\infty} r^{\gamma-1-2s}\,dr.\end{eqnarray*}
Then we use~\eqref{g}
and we obtain the desired result.
\end{proof}

\subsection{Estimates far from the origin}
Now we estimate~${\mathcal{I}}v=
{\mathcal{I}}_1v+{\mathcal{I}}_2v$
at infinity:

\begin{lemma}\label{lemma P3}
Let~$\gamma$ be as in~\eqref{g}
and~$v:\R^n\rightarrow\R$ be a measurable
function such that~$v(x)\le |x|^\gamma$ for any~$x\in\R^n$.

Assume also that~$v(x)=|x|^\gamma$ for any~$x\in\R^n\setminus B_1$.
Then, for any~$x\in \R^n\setminus B_1$,
$$ {\mathcal{I}}v(x)\le C,$$
for some~$C>0$ possibly depending on~$n$, $s$,
$\Lambda$ and~$\gamma$.
\end{lemma}

\begin{proof} Fix~$x\in \R^n\setminus B_1$. Then~$v(x)=|x|^\gamma$.
Moreover~$v(x\pm y)\le|x\pm y|^\gamma$, and so
$$ v(x+y)+v(x-y)-2v(x)\le |x+y|^\gamma+|x- y|^\gamma
-2|x|^\gamma.$$
Therefore, calling~$\omega:=x/|x|$ and changing variable~$r:=|x|\varrho$,
we have that
\begin{equation}\label{su}
\begin{split}
{\mathcal{I}}v(x) \, &=
\int_{\XXX}\big(v(x+\vartheta r)+v(x-\vartheta r)-2v(x)\big)\ddd
\\&\le
\int_{\XXX}\big(|x+\vartheta r|^\gamma+|x- \vartheta r|^\gamma-2|x|^\gamma\big)\ddd \\
&=|x|^{\gamma-2s} \int_{\XXX}\big(|\omega+
\vartheta \varrho|^\gamma+|\omega- \vartheta \varrho|^\gamma-2\big) 
\,\frac{d\mu(\vartheta)\; d\varrho}{|\varrho|^{1+2s}}
\\ &=|x|^{\gamma-2s}
\int_{\XXX} \frac{ 
g(\vartheta \varrho)+g(-\vartheta \varrho)-2g(0)}{|\varrho|^{1+2s}}
\,d\mu(\vartheta)\; d\varrho,
\end{split}\end{equation}
where, for any~$\eta\in\R^n$, we set~$g(\eta):=|\omega+\eta|^\gamma$.
Notice that
\begin{equation}\label{P0}
|g(\eta)|\le (|\omega|+|\eta|)^\gamma =(1+|\eta|)^\gamma.\end{equation}
Moreover~$g\in C^\infty(B_{1/2})$ and, for any~$\eta\in B_{1/2}$
we have that
\begin{eqnarray*}
&&\partial_i g(\eta)=\gamma|\omega+\eta|^{\gamma-2}(\omega_i+\eta_i)\\
{\mbox{and }}&& \partial_{ij}^2 g(\eta)=
\gamma(\gamma-2)|\omega+\eta|^{\gamma-4}(\omega_i+\eta_i)
(\omega_j+\eta_j)+
\gamma|\omega+\eta|^{\gamma-2}\delta_{ij}.
\end{eqnarray*}
Consequently, for any~$\eta\in B_{1/2}$,
$$ |D^2 g(\eta)|\le C_o
|\omega+\eta|^{\gamma-2},$$
for some~$C_o>0$ depending on~$\gamma$ and~$n$,
and~$|\omega+\eta|\ge |\omega|-|\eta|\ge 1/2$, therefore
$$ \|D^2g\|_{L^\infty(B_{1/2})}\le 2^{2-\gamma} C_o.$$
This, together with a Taylor expansion, implies that, for any~$\eta\in B_{1/2}$,
$$ |g(\eta)+g(-\eta)-2g(0)|\le \|D^2g\|_{L^\infty(B_{1/2})}\,|\eta|^2
\le 2^{2-\gamma} C_o\,|\eta|^2.$$
Hence, recalling~\eqref{P0} and~\eqref{spectral}, we obtain that
\begin{eqnarray*}
&& \int_{\XXX} \frac{
g(\vartheta \varrho)+g(-\vartheta \varrho)-2g(0)}{|\varrho|^{1+2s}}
\,d\mu(\vartheta)\; d\varrho\\
&\le& C\,\left(
\int_{S^{n-1}\times(-1/2,1/2)} \frac{|\varrho|^2}{|\varrho|^{1+2s}}
\,d\mu(\vartheta)\; d\varrho+
\int_{S^{n-1}\times(\R\setminus(-1/2,1/2))} \frac{(1+|\varrho|)^\gamma}{
|\varrho|^{1+2s}}
\,d\mu(\vartheta)\; d\varrho\right)
\\ &\le& C\Lambda\,
\left(
\int_{(-1/2,1/2)} \frac{|\varrho|^2}{|\varrho|^{1+2s}}
\,d\varrho
+ \int_{\R\setminus(-1/2,1/2)} \frac{(1+|\varrho|)^\gamma}{
|\varrho|^{1+2s}}
\,d\varrho\right)\\
&\le& C'\Lambda,
\end{eqnarray*}
for some~$C$, $C'>0$, thanks to~\eqref{g}.
We insert this into~\eqref{su} and we obtain the desired estimate
(by possibly renaming the constants).
\end{proof}
 
\section{Construction of an auxiliary barrier}\label{S:B}

Here we use the estimate in Section~\ref{S:T}
and we borrow some ideas from~\cite{DSV} to construct a useful
auxiliary function:

\begin{lemma}\label{barrier}
Let~$\gamma\in(0,2s)$. There exists a function~$v\in C^\infty(\R^n)$ such that
\begin{eqnarray}
&& \label{F.0} v(0)=0,\\
&& \label{F.1} 0\le v(x) \le |x|^\gamma {\mbox{ for any }} x\in\R^n,\\
&& v(x)=|x|^\gamma {\mbox{ if }} |x|\ge 1 \label{F.2} \\
{\mbox{and }}&& 
\sup_{x\in\R^n}{\mathcal{I}}v(x)\le C,
\label{F.3}
\end{eqnarray}
for some~$C>0$.
\end{lemma}

\begin{proof} Let~$\tau\in C^\infty(\R^n)$ be such that~$0\le\tau\le1$
in the whole of~$\R^n$,
$\tau=1$
in~$B_{1/2}$ and~$\tau=0$ in~$\R^n\setminus B_1$.
We define~$v(x):=\big(1-\tau(x)\big)|x|^\gamma$.
In this way, conditions~\eqref{F.0}, \eqref{F.1}
and~\eqref{F.2} are fulfilled.

Furthermore, $v$ satisfies all the assumptions of Lemmata~\ref{lemma P1},
\ref{lemma P2} and~\ref{lemma P3}. Thus, using such results, we obtain
condition~\eqref{F.3}.
\end{proof}

\section{Proof of the main results}\label{P:P}

\begin{proof}[Proof of Theorem~\ref{Liouville}]
The proof relies on a modification of a classical argument
(see for instance~\cite{serrin, handbook}).
In our setting, the barrier constructed in Lemma~\ref{barrier}
will replace (at least from one side) the classical paraboloid.
The details of the argument goes as follows.
Let~$f$, $u$, $K$ and~$\kappa$ as in the statement of
Theorem~\ref{Liouville}. Let~$\gamma:=(2s+\kappa)/2$.
By construction,
\begin{equation}\label{G2}
\gamma\in(\kappa,2s),\end{equation} so we can use
the barrier~$v$ constructed in Lemma~\ref{barrier}.
We fix~$\epsilon>0$ and an arbitrary point~$x_0\in\R^n$,
and we define
\begin{equation}\label{xe}\begin{split}
& w_1(x) := u(x)-u(x_0)+2\epsilon-\epsilon v(x-x_0)\\
{\mbox{and }}\qquad&w_2(x) := u(x)-u(x_0)-2\epsilon+\epsilon v(x-x_0)
\end{split}\end{equation}
We remark that
\begin{eqnarray*}&&
\limsup_{|x|\to+\infty} 
w_1(x)\le\limsup_{|x|\to+\infty} 
[\, u(x)+|u(x_0)|+2\epsilon- \epsilon v(x-x_0)]\\
&&\qquad\le
\limsup_{|x|\to+\infty} 
[ \, K\,(1+|x|^\kappa)+|u(x_0)|+2\epsilon-\epsilon|x-x_0|^\gamma]
=-\infty\\
{\mbox{and }}
&&
\liminf_{|x|\to+\infty} 
w_2(x)\ge\liminf_{|x|\to+\infty}
[ \, u(x)-|u(x_0)|-2\epsilon+ \epsilon v(x-x_0)] \\
&&\qquad\ge
\liminf_{|x|\to+\infty}
[ \, -K\,(1+|x|^\kappa)-|u(x_0)|-2\epsilon+\epsilon|x-x_0|^\gamma]
=+\infty,
\end{eqnarray*}
where we have used~\eqref{K}, \eqref{F.2} and~\eqref{G2}.
As a consequence the maximum of~$w_1$ and the minimum of~$w_2$
are attained, i.e. there exists~$y_1$, $y_2\in\R^n$ such that
\begin{equation}\label{Ex} w_1(y)\le w_1(y_1) \ {\mbox{ and }} \
w_2(y)\ge w_2(y_2) \ {\mbox{ for any $y\in\R^n$. }}\end{equation}
Accordingly, for any~$y\in\R^n$,
\begin{equation}\label{X-1}
\begin{split}
&w_1(y_1+y)+w_1(y_1-y)-2w_1(y_1) \le 0\\
{\mbox{and }}&w_2(y_1+y)+w_2(y_1-y)-2w_2(y_2) \ge 0.
\end{split}\end{equation}
On the other hand
\begin{equation}\label{X-2}
\begin{split}
&w_1(y_1+y)+w_1(y_1-y)-2w_1(y_1)\\ &\quad=
u(y_1+y)+u(y_1-y)-2u(y_1)\\ &\quad\quad-\epsilon \big(v(y_1+y-x_0)+
v(y_1-y-x_0)-2v(y_1-x_0)\big),\\
{\mbox{and }}\quad&
w_2(y_2+y)+w_2(y_2-y)-2w_2(y_2)\\ &\quad=
u(y_2+y)+u(y_2-y)-2u(y_2)\\ &\quad\quad+\epsilon \big(v(y_2+y-x_0)+
v(y_2-y-x_0)-2v(y_2-x_0)\big).
\end{split}\end{equation}
By comparing~\eqref{X-1} and~\eqref{X-2}, we obtain that
\begin{equation}\label{16.bis}\begin{split}
0\,&\ge \int_{\XXX} \big(w_1(y_1+\vartheta r)+w_1(y_1-\vartheta r)-2w_1(y_1)\big)\ddd
\\&={\mathcal{I}}u(y_1)-\epsilon{\mathcal{I}}v(y_1-x_0)\\
{\mbox{and }}\qquad\quad 0\,&\le
\int_{\XXX} \big(w_2(y_2+\vartheta r)+w_2(y_2-\vartheta r)-2w_2(y_2)\big)\ddd
\\&={\mathcal{I}}u(y_2)+\epsilon{\mathcal{I}}v(y_2-x_0).
\end{split}\end{equation}
Therefore, using and~\eqref{U} and~\eqref{F.3}, we obtain that
\begin{equation}\label{Fx}
0\ge f\big(u(y_1)\big)-C\epsilon\qquad
{\mbox{and }}\qquad 0\le f\big(u(y_2)\big)+C\epsilon.
\end{equation}
Now we observe that~$w_1(x_0)=2\epsilon\ge0$
and~$w_2(x_0)=-2\epsilon\le0$, thanks to~\eqref{xe} and~\eqref{F.0}
So, if we
evaluate~\eqref{Ex} at the point~$y:=x_0$,
we obtain that
\begin{equation}\label{90}
0\le w_1(x_0)\le w_1(y_1) \ {\mbox{ and }} \ 0\ge w_2(x_0)\ge w_2(y_2).\end{equation}
Furthermore, using that~$v\ge0$ (recall~\eqref{F.1}), we see from~\eqref{xe}
that
\begin{eqnarray*}
&& w_1(y_1) \le u(y_1)-u(x_0)+2\epsilon
\ {\mbox{ and }} \
w_2(y_2)\ge
u(y_2)-u(x_0)-2\epsilon.\end{eqnarray*}
By comparing this with~\eqref{90}, we conclude
that
$$ {\mbox{$u(y_1)\ge u(x_0)-2\epsilon$ \
and \ $u(y_2)\le u(x_0)+2\epsilon$.}}$$
Therefore, since~$f$ is nondecreasing, we deduce that
$${\mbox{$f\big(u(y_1)\big)\ge
f\big(u(x_0)-2\epsilon\big)$ \
and \ $f\big(u(y_2)\big)\le f\big(u(x_0)+2\epsilon\big)$.}}$$
We plug this information into~\eqref{Fx}, and we obtain that
\begin{equation}\label{188}
0\ge f\big(u(x_0)-2\epsilon\big)-C\epsilon\qquad
{\mbox{and }}\qquad 0\le f\big( u(x_0)+2\epsilon\big)+C\epsilon.\end{equation}
We remark that~$x_0$ was fixed at the beginning and so it is independent of~$\epsilon$
(conversely, the points~$y_1$ and~$y_2$ in general may depend on~$\epsilon$).
This says that we can pass to the limit
as~$\epsilon\to0^+$ in~\eqref{188} and use
the continuity of~$f$ to obtain that
$$ 0\ge f\big(u(x_0)\big)\qquad
{\mbox{and }}\qquad 0\le f\big( u(x_0)\big),$$
that is~$f\big( u(x_0)\big)=0$.
Since~$x_0$ is an arbitrary point of~$\R^n$, we have proved that
\begin{equation}\label{FF}
{\mbox{$f\big( u(x)\big)=0$
for any~$x\in\R^n$.}}\end{equation} 
Thus, using again~\eqref{U}, we obtain that
\begin{equation*}
{\mbox{${\mathcal{I}}u=0$ in~$\R^n$.}}\end{equation*}
{F}rom this and Theorem~2.1 in~\cite{ROS}, we obtain
that~$u$ is a polynomial of degree~$d\in\N$,
with~$d$ less than or equal to the integer part of~$\kappa$.
In particular, $d\le \kappa<2s<2$, hence~$d\in\{0,1\}$
and thus~$u$ is an affine function. So
we can write~$
u(x)=\varpi\cdot x+c$, for some~$\varpi\in\R^n$
and~$c\in\R$.

As a side remark notice that, since~$d\le\kappa$, 
when the additional assumption~$\kappa<1$ holds true,
we have that~$d=0$, and consequently~$\varpi=0$
and~$u$ is constant.
These considerations establish
claim~(ii) in Theorem~\ref{Liouville}.

Now we prove that
\begin{equation}\label{OM-9}
{\mbox{if~$f$ does not vanish identically, then~$\varpi=0$}.}\end{equation}
Indeed, if, by contradiction, we had~$\varpi\ne0$,
given any~$r\in\R$, we can take~$x_\star:= (r-c)|\varpi|^{-2}\varpi$.
Then
$$ u(x_\star) = \varpi\cdot x_\star+c =r,$$
thus, by~\eqref{FF}, we get that~$f(r)=f(u(x_\star))=0$.
Since~$r$ was arbitrary, this would say that~$f$ vanishes identically,
in contradiction with our assumptions. This proves~\eqref{OM-9}.

By~\eqref{OM-9} and~\eqref{FF} we obtain claim~(i) in Theorem~\ref{Liouville}.
This completes
the proof of Theorem~\ref{Liouville}. \end{proof}

\begin{proof}[Proof of Theorem~\ref{Liouville.2}]
The proof of Theorem~\ref{Liouville} goes through in this case, just considering
only the function~$w_1$ (to obtain the first statement
of Theorem~\ref{Liouville.2}), or only the function~$w_2$
(to obtain the second statement). \end{proof}

\vfill

\end{document}